\documentclass[11pt]{article}

\usepackage[utf8]{inputenc}
\usepackage{amsmath, amsfonts, amssymb, amsthm, bbm, cases, cite, csquotes, enumitem, hyperref, mathrsfs, mathtools}

\author{Petr Kosenko}
\title{On a complex-analytic approach to stationary measures on $S^1$ with respect to the action of $PSU(1,1)$}

\newtheorem{definition}{Definition}[section]
\newtheorem{theorem}{Theorem}[section]
\newtheorem{proposition}{Proposition}[section]
\newtheorem{corollary}{Corollary}[section]
\newtheorem{lemma}{Lemma}[section]
\newtheorem{example}{Example}[section]
\newtheorem{conjecture}{Conjecture}[section]

\usepackage[left=2cm,right=2cm,top=2cm,bottom=2cm,bindingoffset=0cm]{geometry}

\newcommand{\Addresses}{{
		\bigskip
		\footnotesize
		\noindent
		\textbf{Petr Kosenko}, \\
		Department of Mathematics, University of British Columbia, 1984 Mathematics Road, Vancouver, BC, Canada, V6T 1Z2 \\	
		\textit{E-mail:} \texttt{pkosenko@math.ubc.ca} \\
		\textit{ORCID ID:} https://orcid.org/0000-0002-4150-0613
}}

\begin{document}

\maketitle

\begin{abstract}
We provide a complex-analytic approach to the classification of stationary probability measures on $S^1$ with respect to the action of $PSU(1,1)$ on the unit circle via M\"{o}bius transformations by studying their Cauchy transforms from the perspective of generalized analytic continuation. We improve upon results of Bourgain and present a complete characterization of Furstenberg measures for Fuchsian groups of first kind via the Brown-Shields-Zeller theorem.
\end{abstract}

\section{Introduction}

\subsection{Background}

One of the most important notions in dynamical systems is of an \textbf{invariant measure}: given a topological space $X$ and a self-map $T : X \rightarrow X$, one can study Borel measures $\mu \in Bor(X)$ on $X$ which satisfy
\begin{equation}
	(T_* \mu)(A) = \mu(T^{-1}(A)) = \mu(A).
\end{equation}

This notion works quite well when provided with a single map $X \rightarrow X$. However, one often encounters nice spaces equipped with group actions $\Gamma \curvearrowright X$, and it is entirely possible that there are no measures invariant with respect to every element $\gamma \in \Gamma$.

Nevertheless, there is a natural weakening of the above definition, requiring a measure to be invariant ``on average''.

\begin{definition}
	Consider a group action $\Gamma \curvearrowright X$. Let $\nu$ be a Borel probability measure on $X$. Given a Borel measure $\mu$ on $\Gamma$, we say that $\nu$ is \textbf{$\mu$-stationary} (with respect to the action) if
	\begin{equation}
		\label{intro: stationary measure}
		\nu = \int_\Gamma  \gamma_* \nu d \mu(\gamma).
	\end{equation}
\end{definition}

It is easy to see that any invariant measure is stationary with respect to any probability measure on $\Gamma$, but the inverse is, of course, not true. Being a $\mu$-stationary measure is, evidently, a much weaker condition. Stationary measures exist in very general settings, unlike invariant ones, but they are no less important, as they are closely related to the structure of harmonic functions and Poisson boundaries, with applications to the long-term dynamics of random walks on groups. 

Given an admissible random walk $(X_n)$ on a non-elementary discrete subgroup $\Gamma \subset PSU(1,1)$ with a finite first moment, we know (due to Furstenberg (\cite{furstenberg71}) and Kaimanovich (\cite{kaimanovich2000poisson})) that $(X_n)$ converges to the Poisson boundary almost surely. The respective pushforward of the resulting measure to $S^1$ via the identification $\partial \Gamma \simeq S^1$ with the Gromov boundary yields a (unique) $\mu$-stationary measure $\nu_\mu$ with respect to the action of $\Gamma$, called the \textbf{hitting measure} of the random walk. A big open problem in measured group theory is to understand when hitting measures are singular or absolutely continuous with respect to the Lebesgue measure on $S^1$. In particular, recall the Kaimanovich-le Prince's singularity conjecture:

\begin{conjecture}[\cite{kaimanovich2011matrix}]
	\label{Fuchsian singularity conjecture}
	For every finite-range admissible random walk $(X_n)$ generated by a probability measure $\mu$ on a cocompact Fuchsian group $\Gamma$, the hitting measure $\nu_\mu$ is singular with respect to the Lebesgue measure on $S^1 \simeq \partial \Gamma$.
\end{conjecture}

This conjecture is known to hold for non-cocompact lattices due to \cite{guivarch1990}, and the author's thesis \cite{mythesis} provides affirmative results for nearest-neighbour random walks on cocompact Fuchsian groups, but the conjecture is still widely open, as we did show that even the recently developed geometric ideas are insufficient to completely settle the conjecture.

In our paper we will study the actions of subgroups $\Gamma \subset G = PSU(1,1)$ induced by the action of $G$ on $S^1$ via hyperbolic isometries, or, more concretely, M\"{o}bius transformations. We aim to present a complex-analytic framework which, as we believe, can unify the majority existing results about stationary measures on $S^1$ with respect to the action of $\Gamma$ and probability measures satisfying a finite first moment condition. Keep in mind that the analysis will be different depending on several factors:

\begin{itemize}
	\item Whether $\mu$ has finite support or not,
	\item If $\mu$ is infinitely supported, the moment conditions on $\mu$ will matter (first moment, exponential moment, superexponential moment, and so on...),
	\item Whether the subgroup of $G$ generated by the support of $\mu$ is discrete or not,
	\item If the generated subgroup is discrete, whether it is of first or second type,
	\item And, finally, if it is of first kind, whether it is cocompact or not.
\end{itemize}

First results about stationary measures and Poisson boundaries for discrete subgroups of $SL_n(\mathbb{R})$ were established by Furstenberg in \cite{furstenberg1963noncommuting}. In particular, the question of when the Lebesgue measure is $\mu$-stationary was first studied by Furstenberg as well in \cite{furstenberg71}. Pure Fourier-like approaches were independently demonstrated by \cite{Bourgain2012} and \cite{MR2969625}, which allow us to study $\mu$-stationary measures for \textbf{dense} subgroups of $PSU(1,1)$. However, their methods do not apply for discrete groups and are quite delicate with respect to the initial data, requiring complicated number-theoretic and analytic methods to properly apply. There have been multiple independently developed improvements to Bourgain's approach, see \cite{lequen2022absolutely} and \cite{kogler2022locallimittheoremrandom} for latest examples, but they still do not apply to the discrete case and non-finite supports. We also want to mention \cite{kittle2023absolutely}, which provides an entirely different analytic framework to study stationary measures, allowing us to consider stationary measures with continuous but not necessarily differentiable densities. Finally, we want to mention recent attempts to understand the structure of harmonic and Patterson-Sullivan measures using thermodynamic formalism, for example, \cite{garcía2023dimension} and \cite{cantrell2022invariant}. Once again, these approaches are not universal, as Garc\`{i}a-Lessa's paper does not generalize to first-kind Fuchsian groups, and the thermodynamic approach of Cantrell-Tanaka provides considerably more information for Patterson-Sullivan measures than harmonic measures.

As one can see, up until now there was no single method which unified all above settings, and until very recently, the general consensus was that no such method should have exist, in light of the incredible variety of techniques used to study different settings. 

\subsection{Main results}

Inspired by the standard techniques used to study affine self-similar measures on $\mathbb{R}^n$, and their respective Parseval frames, papers of R.S.Strichartz (see \cite{strichartzI} and sequels), together with \cite{denseanalytic} and more recent papers of E.Weber and J.Herr \cite{weber2017paleywiener} and \cite{axioms6020007}, we have developed a promising approach which, in theory, could unify many standard results about stationary measures on $S^1$ with respect to the action of $PSU(1,1)$. The idea is to consider an appropriate integral transform on $S^1$ which respects the action of $PSU(1,1)$ and ``preserves'' \eqref{intro: stationary measure}. The Fourier transform is known to not respect this action, as the resulting exponential terms $e^{(az +b) / (cz+d)}$ are difficult to control. The Helgason-Fourier transform seems to be a better candidate, but integrating the powers of the Poisson kernel $(z, \xi) \mapsto \left( \frac{1 - |z|^2}{|z - \xi|^2}\right)^{\lambda i + 1}$ against a stationary measure does not actually preserve \eqref{intro: stationary measure} in a way we want. Essentially, given a $\mu$-stationary measure $\nu$ on $S^1$, one can easily check that the resulting smooth eigenvector of the hyperbolic Laplacian
\[
\psi(z) := \frac{1}{2\pi} \int_{0}^{2\pi} \left( \frac{1 - |z|^2}{|z - e^{i t}|^2}\right)^{\lambda i + 1} d\nu(t)
\]
does not immediately exhibit any nice properties with respect to the action of $PSU(1,1)$, unlike what we see for Patterson-Sullivan measures. Nevertheless, see Section \ref{Criterion for the singularity of the harmonic measure} for a criterion for verifying singularity using harmonic functions on $\mathbb{D}$.

However, replacing the Poisson kernel with its logarithm, which is closely related to the Busemann cocycle, does the trick, turning a multiplicative relation into an additive one. The resulting functional equation \eqref{intro: functional equation} serves as a proper holomorphic version of \eqref{intro: stationary measure}, and, in a way, it allows us to change the perspective, as we shift from the measurable setting to a holomorphic one on $\mathbb{D}$, granting access to the powerful complex-analytic machinery.

\textbf{Remark.} Before stating our results, we would like to point out a standard reduction: it is sufficient to study pure $\mu$-stationary measures due to the fact that the action of $PSU(1,1)$ respects the Lebesgue decomposition.

In order to formulate the main result, we need the definition of the Cauchy transform:
\[
f_\nu(z) := \frac{1}{2\pi} \int_{0}^{2\pi} \dfrac{d\nu(t)}{e^{it} - z}.
\]

\begin{theorem}
	\label{T:main result}
	Let $\mu$ be a probability measure on $G=PSU(1,1)$.
	\begin{itemize}
		\item If a probability measure $\nu$ on $S^1$ is $\mu$-stationary, then
		\begin{equation}
			\label{intro: functional equation}
			\int_G f_\nu(\gamma^{-1}.z)(\gamma^{-1})'(z)  d\mu(\gamma) - f_\nu(z) = \int_G \frac{d \mu(\gamma)}{z - \gamma.\infty},
		\end{equation}
		for every $z \in \mathbb{D}$. 
		\item If, in addition, the pair $(S^1, \nu)$ is the model for the Poisson boundary of $(\Gamma, \mu)$, then a probability measure $\nu$ on $S^1$ is $\mu$-stationary \textbf{if and only if} \eqref{intro: functional equation} holds for all $z \in \mathbb{D}$.
	\end{itemize}
\end{theorem}

The power of this theorem lies in the fact that we managed to successfully transform a measurable functional equation on the circle into a holomorphic condition on the unit disk, which allows us to make use of powerful complex-analytic techniques.

Evidently, we are able to extract the most amount of information from \eqref{intro: functional equation} for countably supported probability measures $\mu$.

\begin{corollary}
	\label{C:main corollary}
	Let $\mu$ be a countably supported probability measure on $PSU(1,1)$.
	\begin{enumerate}
		\item If $\mu$ has finite support, and there is an element $\gamma \in \text{supp} \, \mu$ with $\gamma.0 \ne 0$, then there are no entire solutions to \eqref{intro: functional equation}. In particular, there are no $\mu$-stationary measures with the Fourier series $\nu \sim \sum_{k \in \mathbb{Z}} a_k e^{i k t}$ with $\limsup_{n \rightarrow \infty} |a_k|^{1/k} = 0.$
		\item Assume that $\limsup\limits_{n \rightarrow \infty} \left\| \int_G \dfrac{d \mu^{*n}(\gamma)}{z - \gamma.\infty} \right\|_1 = \infty$, where $||\cdot||_1$ stands for the norm in $H^1(\mathbb{D})$. Then there are no $\mu$-stationary measures with $L^{1+\varepsilon}(S^1, Leb)$-density for any $\varepsilon > 0$.
		\item Suppose that the subgroup $\Gamma \leq G$ generated by the support of $\mu$ is a non-elementary lattice, $(S^1, \nu)$ is a model for the Poisson boundary of $(G, \mu).$ Denoting the bounded harmonic function which represents the identity function on $S^1$ by $\lambda$, we show that the following are equivalent:
		\begin{itemize}
			\item $\sup_{\gamma \in \Gamma} \frac{|\lambda(\gamma) - \gamma.0|}{|\gamma'(0)|} < \infty$
			\item $f_\nu(z) \in H^\infty(\mathbb{D})$.
		\end{itemize}
	\end{enumerate}
\end{corollary}

Equation \eqref{intro: functional equation} gives us quite a lot of insight into the measures $\mu$ for which the (normalized) Lebesgue measure is $\mu$-stationary.

\begin{corollary}
	\label{intro: Lebesgue is stationary}
	Let $\mu$ be a finite Borel measure supported on a discrete subgroup $\Gamma \subset PSU(1,1)$ with a finite first moment. Let's call a measure $\mu$ on $PSU(1,1)$  the \textbf{Furstenberg measure} if the normalized Lebesgue measure on $S^1$ is $\mu$-stationary.
	\begin{enumerate}
		\item The measure $\mu$ is a Furstenberg measure if and only if
		\begin{equation}
			\int_G \frac{d \mu(\gamma)}{z - \gamma.\infty} = 0, \quad |z| < 1.
		\end{equation}
		\item If $\mu$ is a Furstenberg measure, then
		\[
		\limsup_{n \rightarrow \infty} |\mu(\gamma_n)|^{1/n} = 1.
		\]
		\item (Brown-Shields-Zeller) Suppose $\mu$ is a Furstenberg measure. Then $\{\gamma.0\}_{\gamma \in \text{supp} \, \mu}$ is non-tangentially dense in $\mathbb{T}$, which means that $Leb$-\textbf{almost every} point $\xi \in \mathbb{T}$ can be approached by a subsequence $\gamma_n.0$ inside a Stolz angle $\{ z \in \mathbb{D} : \frac{|z - \xi|}{1 - |z|} < \alpha \}$ for some $\alpha > 1$. As a corollary from \cite[Remark 2]{brownsums}, we get
		\[
		\sum_{\gamma \in \text{supp}\, \mu} (1 - |\gamma.0|) = \infty.
		\]
	\end{enumerate}
\end{corollary}

\textbf{Remark.} One fascinating detail of this theorem lies in the fact that the Brown-Shields-Zeller theorem detects the divergence of the Poincar\'e series for first-kind groups at the critical exponent $\delta = 1$, as
\[
\sum_{\gamma} e^{- d_{\mathbb{H}^2}(0, \gamma.0)} = \sum_{\gamma} e^{-\ln\left( \frac{1 + |\gamma.0|}{1 - |\gamma.0|} \right)} \sim \sum_{\gamma} (1 - |\gamma.0|).
\]

Finally, as a corollary from Fatou's theorem, we get a functional-analytic necessary condition for existence of $\mu$-stationary measures with $L^p$-density for $1 < p < \infty$. Before stating the corollary, let us recall the \textbf{Blaschke condition} for a sequence $\{ z_n \} \subset \mathbb{D}$:

\begin{equation}
	\label{Blachke condition}
	\sum_{\gamma \in \text{supp} \, \mu} (1 - |\gamma.0|) < \infty.
\end{equation}

We will say that $\mu$ satisfies the Blaschke condition if and only if $\{ \gamma.0 \}_{\gamma \in \text{supp} \, \mu}$ satisfies \eqref{Blachke condition}.

\begin{corollary}
	\label{functional-analytic necessary condition}
	Let the support of $\mu$ satisfy the Blaschke condition. Then for any $\mu$-stationary measure $\mu$ with $L^p$-density for $1 < p < \infty$, we have
	\[
	f_\nu(z) \in (\overline{T_\mu^*(B_\mu H^q)})^{\perp} \subset H^p,
	\]
	where
	\begin{itemize}
		\item $T_\mu(f) := \sum_{\gamma} \mu(\gamma) (f \circ \gamma^{-1}) (\gamma^{-1})' - f$ is considered as a bounded linear operator \\ $T_\mu : H^p(\mathbb{D}) \rightarrow H^p(\mathbb{D})$, and $\frac{1}{p} + \frac{1}{q} = 1$,
		\item $B_\mu H^q$ denotes the subspace of functions in $H^q$ which vanish on the support of $\mu$. This is a non-trivial closed backward-shift invariant subspace due to the Blaschke condition.
	\end{itemize}
	In particular, if $T_\mu^*(B_\mu H^q)$ is dense in $H^q$, then there are no $\mu$-stationary measures with $L^p$-density.
\end{corollary}

Finally, we would like to formulate another criterion for the singularity of the harmonic measure for random walks on lattices in $PSU(1,1)$.
\begin{theorem}
	\label{criterion}
	Let $\mu$ be a probability measure on a discrete subgroup $\Gamma$ of $PSU(1,1)$, and assume that $(S^1, \mu)$ is the model for the Poisson boundary of $(\Gamma, \mu)$. Then the following statements are equivalent.
	\begin{enumerate}
		\item The harmonic measure $\nu$ is singular.
		\item For $Leb$-almost all $\xi \in \mathbb{T}$ we have
		\[
		\lim_{r \rightarrow 1^-} \lim_{n \rightarrow \infty} \sum_{\gamma} \mu^{* n}(\gamma) \frac{1 - r^2}{|\gamma(0) - r\xi|^2} = 0.
		\]
		\item There exist $|z| < 1$ and $|w| > 1$ such that for $Leb$-a.s. $\xi \in S^1$ the following non-tangential limits exist and are equal to each other:
		\[
		\angle \lim\limits_{\gamma.z \rightarrow \xi} \frac{\lambda_z(\gamma) - \frac{1}{z - \gamma.\infty}}{(\gamma^{-1})'(z)} = \angle \lim\limits_{\gamma.w \rightarrow \xi} \frac{\lambda_w(\gamma) - \frac{1}{w - \gamma.\infty}}{(\gamma^{-1})'(w)}.
		\]
		\item For every $|w| > 1$ and the following non-tangential limits exist and are equal to each other for $Leb$-a.s. $\xi \in S^1$:
		\[
		\angle \lim\limits_{\gamma.0 \rightarrow \xi} \frac{\lambda_0(\gamma) - \overline{\gamma.0}}{(\gamma^{-1})'(0)} = \angle \lim\limits_{\gamma.w \rightarrow \xi} \frac{\lambda_w(\gamma) - \frac{1}{w - \gamma.\infty}}{(\gamma^{-1})'(w)}.
		\]
	\end{enumerate}
\end{theorem}

Corollary \ref{C:main corollary}.1 strictly strengthens the very last remark in \cite{Bourgain2012}, where it was proven that the Lebesgue measure is never stationary with respect to finitely supported measures on $PSU(1,1)$. Corollary \ref{C:main corollary}.2, in theory, provides a purely computational heuristic to showing singularity of stationary measures, as for lattices in $PSU(1,1)$, one expects the poles to converge to $\mathbb{T}$, whereas for dense subgroups one would expect the poles to accumulate inside $\mathbb{D}$, thus forcing the $H^1$-norms to stay bounded. 

Corollary \ref{intro: Lebesgue is stationary} provides several new insights into Furstenberg measures on $PSU(1,1)$. In particular, as our approach deals with signed and complex measures, we are able to talk about complex Furstenberg measures, which is not possible with any geometric approaches. In particular, we obtain Borel sums with poles in the orbits of a non-cocompact lattice which vanishes in $\mathbb{D}$, despite the fact that such counterexamples should be impossible due to Guivarch'-le Jan (see \cite{guivarch1990}). The catch is that the Brown-Shields-Zeller theorem does not control the moments nor the positivity of the resulting coefficients. We also exhibit the first known result restricting the moment conditions of a Furstenberg measure, once again, improving on \cite{Bourgain2012}. The notion of a non-tangential limit seems to be key in this approach. Finally, we remark that studying positive Furstenberg measures should be possible using techniques in \cite{bonsall1989vanishing} and \cite{hayman1990bases}, as they deal with Borel-like series having strictly positive coefficients.

Corollary \ref{functional-analytic necessary condition} provides a pretty significant restriction on the stationary measures in the $L^p$-class for $1 < p < \infty$, and, in theory, the images with respect to the adjoint operator $T_\mu^*$ can be computed explicitly for any measure $\mu$ satisfying the Blaschke condition.

Finally, Theorem \ref{criterion} provides an intrinsic criterion for singularity of the harmonic measure. We would like to highlight that we don't require any special moment conditions for this criterion to work, so it is applicable in a much larger generality than the criterion in \cite{blachere2011harmonic} applied to our case. It is of independent interest to try and relate the above ratios $\frac{\lambda_0(\gamma) - \overline{\gamma.0}}{(\gamma^{-1})'(0)}$ with $\frac{h}{l}$ -- the Hausdorff dimension of the harmonic measure. Unfortunately, our methods do not immediately provide another approach to showing exact-dimensionality for random walks with moment conditions which are worse than superexponential with respect to the hyperbolic distance.

The structure of the paper is as follows. In Section \ref{Preliminaries} we recall all necessary facts about transformations $PSU(1,1)$ and provide a brief recap of complex-analytic tools we are going to use. In Section \ref{Holomorphic stationarity condition} we introduce an appropriate integral transform which fully respects the action of $PSU(1,1)$ to obtain a holomorphic necessary condition for $\mu$-stationarity, thus proving Theorem \ref{T:main result}. In Section \ref{Squeezing water from a stone} we extract the most we can from the resulting equation, using state-of-the-art techniques related to generalized analytic continuations. Finally, in Section \ref{Criterion for the singularity of the harmonic measure} we use the developed machinery to formulate a criterion for the singularity of the harmonic measure in terms of the non-tangential convergence of harmonic functions on lattices.

\subsection{Acknowledgements}
I would like to thank Ilyas Bayramov, Ivan Nikitin, Ilia Binder, Kunal Chawla, Mathav Murugan, Pablo Shmerkin, Lior Silberman and Omer Angel for useful discussions. Also I am tremendously grateful to Alexander Kalmynin, Giulio Tiozzo and Tianyi Zheng for reading the preliminary version of this preprint and for providing continued support and encouragement throughout the past year. Finally, I would like to thank Giulio Tiozzo for organizing a short visit to the Fields institute on Feb 19--22, 2024. 

The author's research is supported by the Natural Sciences and Engineering Research Council of Canada (NSERC).
\section{Preliminaries}
\label{Preliminaries}
\subsection{Everything you need to know about isometries of the disk model of $\mathbb{H}^2$}
In this subsection we will recall basic facts about $PSU(1,1)$ considered as a isometry group of the disk model $\mathbb{D} = \{ |z| < 1 \}$ of the hyperbolic plane. 
\begin{definition}
	\[
	PSU(1,1) = \left\lbrace z \mapsto \frac{az + b}{\overline{b} z + \overline{a}} : a, b \in \mathbb{C}, |a|^2 - |b|^2 = 1 \right\rbrace.
	\]
\end{definition}

From the definition it is evident that every transformation in $PSU(1,1)$ can be represented by a matrix $\begin{pmatrix}
	a & b \\ \overline{b} & \overline{a}
\end{pmatrix}$ (mod scalar matrices). In particular, if $\gamma(z) = \frac{az + b}{\overline{b} z + \overline{a}}$ then $\gamma^{-1}(z) = \frac{\overline{a} z - b}{ -\overline{b}z + a}$.

Also, it will turn out that sometimes working with $\infty$ as a basepoint is more convenient than choosing $0 \in \mathbb{H}^2$, we will use

\begin{equation}
	\label{from inside to outside}
	\overline{\gamma\left( \overline{z}^{-1} \right)} = \frac{1}{\gamma(z)}, \quad z \in \overline{\mathbb{C}},
\end{equation}

and, as a simple corollary,

\begin{equation}
	\label{zero-infty}
	\gamma.\infty = \frac{a}{\overline{b}} = \left( \frac{\overline{b}}{a}\right)^{-1} = (\overline{\gamma.0})^{-1}.
\end{equation}

\begin{lemma}
	Let $\gamma(z) = \frac{az + b}{\overline{b} z + \overline{a}}$. then
	\[
	\gamma'(z) = \frac{1}{(\overline{b} z + \overline{a})^2}.
	\]
\end{lemma}

\begin{proof}
	\[
	\gamma'(z) = \frac{a ( \overline{b} z + \overline{a} ) - (az + b)\overline{b}}{(\overline{b} z + \overline{a})^2} = \frac{|a|^2 - |b|^2}{(\overline{b} z + \overline{a})^2} = \frac{1}{(\overline{b} z + \overline{a})^2}.
	\]
\end{proof}

Recall that for any $\gamma(z) = \frac{a z + b}{\overline{b} z + \overline{a}}$ with $|a|^2 - |b|^2 = 1$ we have
\begin{equation}
	\label{poles of the log derivative}
		\frac{1}{2} \frac{\gamma''(z)}{\gamma'(z)} = \frac{1}{2} \frac{-2\overline{b}}{(\overline{b}z + \overline{a})^3} \left( \frac{1}{(\overline{b}z + \overline{a})^2}\right)^{-1} = \frac{- \overline{b}}{\overline{b}z + \overline{a}} = - \frac{1}{z + \frac{\overline{a}}{\overline{b}}} = - \frac{1}{z - \gamma^{-1}.\infty}.
\end{equation}

Finally, we also will require the following proposition.

\begin{proposition}
	\label{weighted composition ops}
	Let $\gamma \in PSU(1,1)$ and consider the linear operator
	\[
	V_\gamma(f)(z) := f(\gamma^{-1}(z)) (\gamma^{-1})'(z).
	\] 
	Then for every $w \in \mathbb{C}$ we have
	\[
	V_\gamma \left( \frac{1}{w - z} \right) = \frac{1}{\gamma.w - z} - \frac{1}{\gamma.\infty - z}.
	\]
\end{proposition}
\begin{proof}
	We can prove it via a direct computation.
	\[
	V_\gamma \left( \frac{1}{w - z} \right) = \frac{(\gamma^{-1})'(z)}{w - \gamma^{-1}.z}.
	\]
	If we denote $\gamma^{-1}.z = \frac{\overline{a} z - b}{-\overline{b}z + a}$, then we get
	\[
	\begin{aligned}
		\frac{(\gamma^{-1})'(z)}{w - \gamma^{-1}.z} 
		&= \frac{1}{(- \overline{b}z + a)^2}\left( \frac{1}{w - \frac{\overline{a} z - b}{-\overline{b}z + a}}\right) = \frac{1}{(-\overline{b} z + a)(-\overline{b} z w + a w - \overline{a} z + b)} = \\
		&= \frac{1}{(-\overline{b} z + a)(-z (\overline{b} w + \overline{a}) + a w + b)} =  \frac{1}{(-\overline{b} z + a)(\overline{b} w + \overline{a})(-z  + \gamma.w)} = \\ 
		&=\frac{\frac{1}{\overline{b}^2 w + \overline{a}\overline{b}}}{(\frac{a}{\overline{b}} - z)(-z  + \gamma.w)} = \frac{\frac{a}{\overline{b}} - \frac{aw + b}{\overline{b}w + \overline{a}}}{(\gamma.\infty - z)(\gamma.w - z)} = \frac{1}{\gamma.w - z} - \frac{1}{\gamma.\infty - z}.
	\end{aligned}
	\]
\end{proof}

\subsection{Random walks on groups}
Here we present standard background on random walks, we refer to \cite{furstenberg1963noncommuting}, \cite{KV83}, \cite{kaimanovich2000poisson} for more details.

\begin{definition}
	Let $\mu$ be a probability measure on a group $G$. A random walk on $G$ induced by $\mu$ is a collection of $G$-valued random variables $(X_n)_{n \ge 1}$,
	\[
	X_n = g_1 \dots g_n, 
	\]
	where $g_i$ are i.i.d $\mu$-distributed $G$-valued random variables.
\end{definition}
\textbf{Remark.} We assume that our random walks start from identity unless explicitly mentioned otherwise.

\begin{definition}
	Given a probability measure $\mu$ on $G$, we say that a function $\varphi : G \rightarrow \mathbb{C}$ is $\mu$-harmonic if
	\[
	\int_G \varphi(g \gamma) d\mu(\gamma) = \varphi(g), \quad g \in G.
	\]
\end{definition}

We will say that a random walk is non-degenerate if the support of $\mu$ generates $G$ as a semigroup.

One way to understand bounded $\mu$-harmonic function is by considering the \textbf{Poisson boundary} $(\partial_{Pois} G, \nu)$ of the random walk $(X_n)$. We are not going to present all possible ways to construct the Poisson boundary, so we will use the following idea. Suppose that $(G, \mu)$ acts on a Borel space $(X, \nu)$ in such a way that $\nu$ is a $\mu$-stationary measure:
\[
\nu = \int_{G} \gamma_* \nu d\mu(\gamma).
\]
Then it makes sense to consider the following correspondence:
\begin{equation}
	\label{Poisson harmonic correspondence}
	L^\infty(X, \nu) \ni f \mapsto \left( g \mapsto \int_G f(x) dg_* \nu(x) \right) \in Har^\infty(G, \mu).
\end{equation}
It is not difficult to see that the resulting function on $G$ is always a bounded $\mu$-harmonic function. This motivates a definition for the ``harmonic'' model of the Poisson boundary of $(G, \mu)$.
\begin{definition}
	Let $(X, \nu)$ be a Borel space equipped with a $\mu$-stationary probability measure $\nu$ with respect to an action of $(G, \mu)$. Then we way that $(X, \nu)$ is a (harmonic) model for the Poisson boundary of $(G, \mu)$ if \eqref{Poisson harmonic correspondence} is an isometric isomorphism.
\end{definition}

In our paper we will focus on measures $\mu$ on $PSU(1,1)$ with respect to which the unit circle $S^1 \simeq \mathbb{T} = \{ |z| = 1 \}$ will be a model for the Poisson boundary, equipped with a suitable $\mu$-stationary probability measure $\nu$, which we will refer to as the (unique) \textbf{harmonic measure}. We won't focus too much on the conditions on $\mu$ which make $(S^1, \nu)$ the (unique) Poisson boundary, as this is an area of research with a very long history. However, keeping in mind that the lattices of $PSU(1,1)$ are hyperbolic groups, we would like to mention the latest results in \cite{chawla2022poissonboundaryhyperbolicgroups}, which ensures that the Poisson boundary coincides with the Gromov boundary with very minimal restrictions on $\mu$.

\textbf{Remark.} Indeed, the condition of $\mu$-stationarity can be interpreted as self-similarity with respect to the action of $PSU(1,1)$ on the unit circle. We have to be a bit careful because the M\"{o}bius maps in the corresponding IFS are not contractions, though. 

Finally, let us restate the Kaimanovich-le Prince's singularity conjecture:
\begin{conjecture}[\cite{kaimanovich2011matrix}]
	For every finitely supported non-degenerate random walk $(X_n)$ generated by a probability measure $\mu$ on a cocompact Fuchsian group $\Gamma$, the harmonic measure $\nu$ is singular with respect to the Lebesgue measure on $S^1 \simeq \partial \Gamma$.
\end{conjecture}

Our goal will be to develop approaches to the above conjecture via complex-analytic techniques, mainly relying on the properties of composition operators on Hardy spaces and the structure of special closed subspaces of the Hardy spaces $H^p(\mathbb{D})$.

\textbf{Remark.} This conjecture can be viewed as an opposite ``twin'' of the (still open) Bernoulli convolution problem, which is concerned with showing the absolute continuity of a certain family of measures on $\mathbb{R}$ which are stationary with respect to the semigroup generated by maps $x \mapsto \lambda x \pm 1$, $\lambda \in (1/2, 1)$. 

\subsection{Complex-analytic prerequisites: Hardy spaces}
We will heavily rely on standard complex-analytic techniques related to Cauchy transforms and generalized analytic continuation, we refer to standard textbooks on these topics: \cite{Shapiro1968}, \cite{cimahardy}, \cite{book:738388}.

Let us denote $\mathbb{D} = \{ |z| < 1 \}$ and $\mathbb{D}_e := \overline{\mathbb{C}} \setminus \overline{\mathbb{D}}$. Given a domain $U \subset \overline{\mathbb{C}}$, we will denote the space of holomorphic functions on $U$ by $\mathfrak{H}(U)$ and the space of meromorphic functions on $U$ by $\mathfrak{M}(U)$.

\begin{definition}
	Let $0 < p < \infty$. The \textbf{Hardy space} $(H^p(\mathbb{D}), || \cdot ||_p)$ is a space of holomorphic functions on $\mathbb{D}$ defined as follows.
	\[
	H^p(\mathbb{D}) = \left\lbrace  f \in \mathfrak{H}(\mathbb{D}) \ | \ ||f||_p := \sup_{0 < r < 1} \left( \frac{1}{2 \pi} \int_{0}^{2 \pi} |f(r e^{it})|^p dt \right)^{1/p} < \infty \right\rbrace.
	\] 
	If $p = \infty$, then we define $(H^\infty(\mathbb{D}), ||f||_\infty)$ as the space of bounded holomorphic functions on $\mathbb{D}$ equipped with the sup-norm.
	
	Finally, we define $H^p(\mathbb{D}_e) := \{ z \mapsto f(1/z) : f \in H^p(\mathbb{D})\}$, with $H^p_0(\mathbb{D}_e) \subset H^p(\mathbb{D}_e)$ denoting functions vanishing at infinity.
\end{definition} 

It is well-known that for $1 \le p \le \infty$ the function $||\cdot||_p : H^p(\mathbb{D}) \rightarrow \mathbb{R}_{\ge 0}$ defines a norm, so the respective Hardy spaces $H^p(\mathbb{D})$ are Banach spaces for $1 \le p \le \infty$. For $0 < p < 1$ the Hardy spaces $H^p(\mathbb{D})$ admit a complete translation-invariant metric defined by $d(f, g) := ||f - g||^p_p$, but the topology it defines is not non-locally convex.

\begin{definition}
	\label{D:non-tangential convergence}
	A sequence of points $\{z_n\} \subset \mathbb{D}$ is said to \textbf{non-tangentially} converge to $\xi \in \partial \mathbb{D}$ if there exists a \textbf{Stolz angle} $A = \{ \frac{|\xi - z|}{1 - |z|} \le M \}$ and $N > 0$ such that $z_n \rightarrow \xi$ and $z_n \subset A$ for $n > N$.
\end{definition}

We will frequently use the following classical theorems.

\begin{theorem}[Fatou's theorem]
	\label{T: Fatou theorem}
	Every holomorphic function $f \in H^p(\mathbb{D})$ for $0 < p \le \infty$ admits a non-tangential limit $f(\zeta)$ for $Leb$-almost every $\zeta \in S^1$ which belongs to $L^p(S^1, Leb)$. Moreover, for $0 < p < \infty$ we have
	\[
	||f||_p = \left( \frac{1}{2 \pi} \int_{0}^{2 \pi} |f(e^{i t})|^p dt\right)^{1/p} .
	\]
\end{theorem}

\begin{theorem}[F. Riesz, M. Riesz]
	\label{T: Riesz-Riesz theorem}
		For $p \ge 1$ we have a complete realization of the Hardy spaces $H^p(\mathbb{D})$ as subspaces $L^p(S^1)$: these are exactly the functions with vanishing negative Fourier coefficients.
\end{theorem}

Let us briefly list some examples of holomorphic functions in Hardy spaces.

\begin{example}
	\indent
	\begin{itemize}
		\item Analytic polynomials $p(z) = a_0 + \dots + a_n z^n$ are dense in $H^p(\mathbb{D})$ for all $0 < p < \infty$, and are $wk^*$-dense in $H^\infty(\mathbb{D})$. (\cite[Theorem 1.9.4]{book:738388})
		\item If $0 < p < q \le \infty$, then $H^q(\mathbb{D}) \subsetneq H^p(\mathbb{D})$.
		\item For any $z_0 \in \mathbb{D}_e$ and $k > 0$, we have $\frac{1}{(z - z_0)^k} \in H^\infty(\mathbb{D})$, hence $\frac{1}{(z - z_0)^k} \in H^p(\mathbb{D})$ for any $0 < p \le \infty$. Keep in mind that this will not hold for any $|z_0| = 1$, see later examples.
	\end{itemize}
\end{example}

\subsection{Complex-analytic prerequisites: pseudocontinuations}

\begin{definition}
	Let $f$ be a meromorphic function on $\mathbb{D}$. If there exists a function $T_f$ which is meromorphic on $\mathbb{D}_e$ such that the non-tangential limits of $f$ and $\tilde{f}$ exist on $\mathbb{T}$ and coincide Leb-almost everywhere, then we say that $f$ is \textbf{pseudocontinuable}, and $\tilde{f}$ is a \textbf{pseudocontinuation} of $f$, and vice versa.
\end{definition}

In our paper we will use several important results about non-tangential limits and pseudocontinuations.

\begin{theorem}[Lusin-Privalov, \cite{privalov1956randeigenschaften}]
	If $f$ is pseudocontinuable, then its pseudocontinuation is unique.
\end{theorem}

As a corollary, we get that pseudocontinuations are compatible with analytic continuations.

\begin{definition}
	Let $\nu$ be a complex finite Borel measure on $S^1$. Then its \textbf{Cauchy-Szeg\H{o} transform} is the integral
	\begin{equation}
		C_\nu(z) := \frac{1}{2\pi} \int_0^{2 \pi} \dfrac{d \nu(t)}{1 - e^{-it} z}.
	\end{equation}
	Its \textbf{Cauchy transform} is the integral
	\begin{equation}
		f_\nu(z) := \frac{1}{2\pi} \int_0^{2 \pi} \dfrac{d \nu(t)}{e^{it} - z}.
	\end{equation}
\end{definition}

It is easy to see that $C_\nu(z) = f_{\nu'}(z)$ for $d \nu'(t) = e^{it} d \nu(t)$, but we will still use both transforms when convenient.

The properties of $C_\nu(z)$ as a holomorphic function on $\mathbb{D}$ strongly depend on $\nu$ itself, but the following theorem of Smirnov ensures that we at least end up in $H^p(\overline{\mathbb{C}}  \setminus \mathbb{T})$ for $p < 1$.

\begin{theorem}[Smirnov]
	\label{T:Smirnov}
	Let $f(z) = C_\nu(z)$ for some complex finite Borel measure $\nu$. Then $f\vert_{\mathbb{D}} \in H^p(\mathbb{D})$ and $f\vert_{\mathbb{D}_e} \in H^p(\mathbb{D}_e)$ for all $0 < p < 1$.
\end{theorem}

We can do better if we know that $\nu \ll Leb$ due to a theorem of M. Riesz.

\begin{theorem}[Riesz]
	\label{Cauchy of Lp}
	Let $\nu$ be an absolutely continuous measure on $S^1$ with the $L^p$-density for $1 < p < \infty$. Then $C_\nu(z) \in H^p(\mathbb{D})$.
\end{theorem}

\textbf{Remark.} This theorem cannot not hold for $p = 1$ or $p = \infty$, as it is well-known that there are no continuous projections $L^1 \rightarrow H^1$ and $L^\infty \rightarrow H^{\infty}.$

Finally, the Cauchy transform of a positive Borel measure on $S^1$ is unique in the following sense:

\begin{theorem}[\cite{book:738388}, Corollary 4.1.3, Proposition 4.1.4]
	Let $\nu_1, \nu_2$ be two probability measures on $S^1$. \\ Then $C_{\nu_1} = C_{\nu_2}$ if and only if $f_{\nu_1} = f_{\nu_2}$ if and only if $\nu_1 = \nu_2$.
\end{theorem}

\begin{example}
	\indent
	\begin{itemize}
		\item Consider $f(z) = (1 - z)^{-1}$. This is the Cauchy transform of the Dirac delta measure $\delta_1$, so $f(z) \in H^p(\mathbb{D})$ for all $0 < p < 1$, but $f(z) \notin H^{1}(\mathbb{D})$. The idea is that the integral
		\[
		\int_{0}^1 \frac{dx}{x^p}
		\]
		converges for $p < 1$ and diverges for $p = 1$.
		\item (\cite{Ale79}) This suggests that for $0 < p < 1$ it makes sense to talk about the closure of all simple poles $z \rightarrow \frac{1}{1 - e^{it}z}$ on $\mathbb{T}$, and it turns out that this closure admits a very nice description:
		\begin{equation}
			\overline{\text{span}}^{H^p} \left\lbrace \frac{1}{1 - \xi z} : |\xi| = 1 \right\rbrace := H^p \cap \overline{H^p_0},
		\end{equation}
		where by $H^p \cap \overline{H^p_0}$ we denote the subspace of all functions $f(z) \in \mathcal{H}(\overline{\mathbb{C}} \setminus \mathbb{T})$, such that both inner and outer components lie in $H^p$, $f(\infty) = 0$ and inner + outer non-tangential limits a.e. exist and a.e. coincide.
	\end{itemize}
\end{example}

The above space $H^p \cap \overline{H^p_0}$ is very important because of the following theorem.

\begin{theorem}[Fatou]
	\label{Fatou theorem}
	Let $\nu$ be a finite complex Borel measure on $\mathbb{T}$. Denote the absolutely continuous part of $\nu$ by $F(\xi)$. Then
	\[
	\lim_{r \rightarrow 1^-} \frac{1}{2\pi} \int_{0}^{2 \pi} \frac{1 - r^2}{|e^{it} - r e^{i\theta}|^2} d\nu(t) = F(e^{i \theta})
	\]
	for $Leb$-almost every $e^{i \theta} \in \mathbb{T}$. 
\end{theorem}

As a simple corollary, we obtain that $\nu$ is a singular measure if and only if the inner and outer components of its Cauchy transform are pseudocontinuations of each other.
\section{Holomorphic stationarity condition}
\label{Holomorphic stationarity condition}
In this section we will provide proofs of the main results.
\begin{theorem}
	Let $\mu$ be a probability measure on $G=PSU(1,1)$. Consider $z \in \overline{\mathbb{C}} \setminus \mathbb{T}$ for which the integral $\int_G \frac{d \mu(\gamma)}{z - \gamma.\infty}$ converges absolutely. Then
	\begin{equation}
		\label{main equation, simp}
		\int_G f_\nu(\gamma^{-1}.z)(\gamma^{-1})'(z)  d\mu(\gamma) - f_\nu(z) = \int_G \frac{d \mu(\gamma)}{z - \gamma.\infty}.
	\end{equation}
\end{theorem}

\begin{proof}
	Due to Proposition \ref{weighted composition ops} we have
	\[
	\begin{aligned}
		T_\gamma(f_{\nu})(z) = T_\gamma \left( \frac{1}{2 \pi} \int_{0}^{2 \pi} \frac{d\nu(t)}{e^{it} -  z} \right) &= \frac{1}{2 \pi}  \int_{0}^{2\pi} \left(  \frac{1}{\gamma.e^{it} -  z} - \frac{1}{\gamma.\infty - z} - \frac{1}{e^{it} -  z} \right)  d \nu(t) = \\ 
		&= \frac{1}{2 \pi}  \int_{0}^{2\pi} \frac{d \gamma_* \nu(t)}{e^{it} -  z}  -  \frac{1}{2 \pi}  \int_{0}^{2\pi} \frac{d \nu(t)}{e^{it} -  z}  - \frac{1}{\gamma.\infty - z} .
	\end{aligned}
	\]
	Fubini's theorem applies since
	\[
	\int_{G \times [0, 2\pi]} \left| \frac{1}{\gamma.e^{it} - z} \right| d\mu(\gamma) d\nu(t) \le \int_{G \times [0, 2\pi]} \frac{1}{|1 - |z||} d\mu(\gamma) d\nu(t) < \infty.
	\]
	Therefore,
	\[
	\begin{aligned}
		& \int_G f_\nu(\gamma^{-1}.z)(\gamma^{-1})'(z)  d\mu(\gamma) - f_\nu(z) = \int_G T_\mu(f_\nu)(z)  d\mu(\gamma) = \int_G \left( \frac{1}{2 \pi} \int_{0}^{2 \pi} \frac{d(\gamma_* \nu - \nu)(t)}{e^{it}-z} - \frac{1}{\gamma.\infty - z}\right) d\mu(\gamma) = \\ &= \frac{1}{2 \pi} \int_{G} \int_{0}^{2\pi} \frac{d(\gamma_* \nu - \nu)(t)}{e^{it}-z} d\mu(\gamma) - \int_{G}\frac{1}{\gamma.\infty - z} d\mu(\gamma) = \frac{1}{2 \pi} \int_{0}^{2\pi} \int_{G}  \frac{d(\gamma_* \nu - \nu)(t)}{e^{it}-z} d\mu(\gamma) - \int_{G}\frac{1}{\gamma.\infty - z} d\mu(\gamma) = \\ 
		&= \int_{G}\frac{1}{z - \gamma.\infty} d\mu(\gamma).
	\end{aligned}
	\]
	which is precisely \eqref{main equation, simp}.
\end{proof}

Before stating the converse result, let us demonstrate a simple example.

\noindent \textbf{Example.} Let $\gamma \in PSU(1,1)$ be a non-elliptic element, and consider $\mu = \delta_{\gamma}$. Let us try to find all $\mu$-stationary measures $\nu$ using our functional equation. From \eqref{main equation, simp} we have
	\[
	f_\nu(\gamma^{-1}.z)(\gamma^{-1})'(z) - f_\nu(z) = \frac{1}{z - \gamma.\infty}.
	\]
	Observe that for any given $z_0 \in \mathbb{D}$ the function $n \mapsto f_\nu(\gamma^{-n}.z_0)(\gamma^{-n})'(z_0) - \frac{1}{z_0 - \gamma^n.\infty}$ is a bounded harmonic function on $\mathbb{Z}$.
	It is a harmonic function because
	\[
	\begin{aligned}
		& f_\nu(\gamma^{-n-1}.z_0)(\gamma^{-n-1})'(z_0) - \frac{1}{z_0 - \gamma^{n+1}.\infty} = f_\nu(\gamma^{-1}(\gamma^{-n}.z_0))(\gamma^{-1})'(\gamma^{-n}(z_0)) (\gamma^{-n})'(z_0) - \frac{1}{z_0 - \gamma^{n+1}.\infty} = \\ 
		&= \frac{(\gamma^{-n})'(z_0)}{\gamma^{-n}.z_0 - \gamma.\infty} + f_\nu(\gamma^{-n}.z_0) (\gamma^{-n})'(z_0) - \frac{1}{z_0 - \gamma^{n+1}.\infty} = f_\nu(\gamma^{-n}.z_0)(\gamma^{-n})'(z_0) - \frac{1}{z_0 - \gamma^n.\infty}.
	\end{aligned}
	\]
	Moreover, it is bounded, because
	\begin{equation}
		\label{bounded1}
		f_\nu(\gamma^{-n}.z_0)(\gamma^{-n})'(z_0) = \frac{1}{2 \pi} \int_0^{2 \pi} \frac{(\gamma^{-n})'(z_0) d\nu(t)}{e^{i t} - \gamma^{-n}.z_0} = \left( \frac{1}{2 \pi} \int_0^{2 \pi} \frac{d\nu(t)}{\gamma^n.e^{i t} - z_0}\right)  - \frac{1}{\gamma^n.\infty - z_0},
	\end{equation}
	so
	\[
	|f_\nu(\gamma^{-n}.z_0)(\gamma^{-n})'(z_0)| = \left| \left( \frac{1}{2 \pi} \int_0^{2 \pi} \frac{d\nu(t)}{\gamma^n.e^{i t} - z_0}\right)  - \frac{1}{\gamma^n.\infty - z_0} \right| \le \frac{2}{1 - |z_0|}.
	\]
 	Therefore, for all $n \in \mathbb{Z}$ there exists a constant $C(z_0)$ such that
	\[
	f_\nu(\gamma^{-n}.z_0)(\gamma^{-n})'(z_0) = \frac{1}{z_0 - \gamma^n.\infty} + C(z_0).
	\]
	In particular, $f_\nu(z_0) = C(z_0)$. We want to study both terms as $n \rightarrow \infty$. Due to the dominated convergence theorem we have
	\[
	\lim\limits_{n \rightarrow \infty} f_\nu(\gamma^{-n}.z_0)(\gamma^{-n})'(z_0) = \frac{\nu(\{\gamma_\infty\}) - 1}{\gamma_\infty - z_0} + \frac{\nu(\{\gamma_{-\infty}\})}{\gamma_{-\infty} - z_0}  = \frac{1}{z_0 - \gamma_\infty} + C(z_0),
	\]
	where $\gamma_{\pm \infty} = \lim_{n \rightarrow \pm \infty} \gamma^n.\infty \in S^1$.
	This is equivalent to
	\[
	f_\nu(z_0) = C(z_0) = \frac{\nu(\{\gamma_\infty\})}{\gamma_\infty - z_0} + \frac{\nu(\{\gamma_{-\infty}\})}{\gamma_{-\infty} - z_0}.
	\]
	But this implies that $\nu$ is exactly a convex combination of two delta-measures at $\gamma_\infty$ and $\gamma_{-\infty}$.
We can adapt the same strategy to obtain the following result.
\begin{theorem}
	\label{converse}
	Let $\mu$ be a probability measure on $G=PSU(1,1)$. Assume that the unit circle $(S^1, \nu)$ is a model for the Poisson boundary of $(G, \mu)$.	If $f_{\tilde{\nu}}$ satisfies \eqref{main equation, simp} for all $z$ in $\mathbb{D}$, then $\tilde{\nu} = \nu$.
\end{theorem}

\begin{proof}
	Fix $z_0 \in \mathbb{D}$ and observe that \eqref{main equation, simp} is equivalent to saying that 
	\begin{equation}
		\lambda_{z_0} : G \rightarrow \mathbb{C}, \quad \gamma \mapsto f_{\tilde{\nu}}(\gamma^{-1}.z_0)(\gamma^{-1})'(z_0) - \frac{1}{z_0 - \gamma.\infty}
	\end{equation}
	is a $\mu$-harmonic function on $G$:
	\[
	\begin{aligned}
		& \int_G  \left(f_{\tilde{\nu}}((\gamma h)^{-1}.z_0)((\gamma h)^{-1})'(z_0) - \frac{1}{z_0 - (\gamma h).\infty}\right) d\mu(h)  = \\
		&= \int_G  \left(f_{\tilde{\nu}}(h^{-1}.(\gamma^{-1}.z_0))(h^{-1})'(\gamma^{-1}.z_0) (\gamma^{-1})'(z_0) - \frac{1}{z_0 - (\gamma h).\infty}\right) d\mu(h) = \\
		&= \left( \int_G  f_{\tilde{\nu}}(h^{-1}.(\gamma^{-1}.z_0))(h^{-1})'(\gamma^{-1}.z_0) d\mu(h) \right) (\gamma^{-1})'(z_0) - \int_G \frac{d\mu(h)}{z_0 - (\gamma h).\infty} = \\
		&= f_{\tilde{\nu}}(\gamma^{-1}.z_0) (\gamma^{-1})'(z_0) + \int_G \frac{d\mu(h) (\gamma^{-1})'(z_0)}{\gamma^{-1}.z - h.\infty} - \int_G \frac{d\mu(h)}{z_0 - (\gamma h).\infty} = \\
		&= f_{\tilde{\nu}}(\gamma^{-1}.z_0) (\gamma^{-1})'(z_0) - \frac{1}{z_0 - \gamma.\infty}.
	\end{aligned}
	\] 
	This is a bounded function due to \eqref{bounded1}:
	\[
		|f_{\tilde{\nu}}(\gamma^{-1}.z_0)(\gamma^{-1})'(z_0)| = \left| \left( \frac{1}{2 \pi} \int_0^{2 \pi} \frac{d\tilde{\nu}(t)}{\gamma.e^{i t} - z_0}\right)  - \frac{1}{\gamma.\infty - z_0} \right| \le \frac{2}{1 - |z_0|}.
	\]
	Recall that a bounded harmonic function has limits $\nu$-a.s along sample paths $\gamma_n$ which converge to $S^1$. Let us pick such a path $\gamma_n \rightarrow \xi \in S^1$.
	Then, as in the above argument, we will end up with
	\[
	\lim\limits_{n \rightarrow \infty} \lambda_{z_0}(\gamma_n) = \frac{1}{\xi - z_0}.
	\]
	Combined with the Poisson representation for bounded harmonic functions, this yields $\lambda(id) = f_{\tilde{\nu}}(z_0) = \frac{1}{2 \pi} \int_0^{2 \pi} \frac{d \nu(t)}{e^{it} - z_0} = f_\nu(z_0)$.
\end{proof}

To formulate stronger results, we also need to consider holomorphic solutions to \eqref{main equation, simp} outside of the unit disk.

\begin{proposition}
	Let $\mu$ be a probability measure, and let $f(z)$ be a holomorphic solution to \eqref{main equation, simp} in $\mathbb{D}$. Then the function $g(z) := -\frac{\overline{f(\overline{z}^{-1})}}{z^2} - \frac{1}{z}, z \in \mathbb{D}_e$ is a holomorphic solution to \eqref{main equation, simp} outside of the unit disk for all $z \in \mathbb{D}_e$ such that the integral $\int_G \frac{d \mu(\gamma)}{z - \gamma.\infty}$ converges absolutely.
\end{proposition} 
\begin{proof}
	First of all, we observe that
	\begin{equation}
		\overline{(\gamma^{-1})'(\overline{z}^{-1})} = \frac{z^2}{(\gamma^{-1}.z)^2} (\gamma^{-1})'(z), \quad \gamma \in PSU(1,1), z \in \mathbb{D}.
	\end{equation}
	This follows from a direct computation:
	\[
		\overline{(\gamma^{-1})'(\overline{z}^{-1})} = \frac{1}{\overline{(-\overline{b} \overline{z}^{-1} + a)^2}} = \frac{z^2}{(-b + \overline{a}z)^2} = \frac{z^2}{(-\overline{b}z +a)^2} \frac{(-\overline{b}z +a)^2}{(\overline{a}z - b)^2} = \frac{z^2}{(\gamma^{-1}.z)^2} (\gamma^{-1})'(z).
	\]
	Using the above lemma, we rewrite \eqref{main equation, simp} as follows.
	\[
	\begin{aligned}
		& \int_{G} f(\gamma^{-1}.z)(\gamma^{-1})'(z) d\mu(\gamma) - f(z) = \int_G \frac{d \mu(\gamma)}{z - \gamma.\infty}, \quad |z| < 1 \iff \\ 
		& \iff \int_{G} f(\gamma^{-1}.\overline{z}^{-1})(\gamma^{-1})'(\overline{z}^{-1}) d\mu(\gamma) - f(\overline{z}^{-1}) = \int_G \frac{d \mu(\gamma)}{\overline{z}^{-1} - \gamma.\infty}, \quad |z| > 1, \iff \\
		& \iff \int_{G} \overline{f(\gamma^{-1}.\overline{z}^{-1})}\overline{(\gamma^{-1})'(\overline{z}^{-1})} d\mu(\gamma) - \overline{f(\overline{z}^{-1})} = \int_G \frac{d \mu(\gamma)}{z^{-1} - \overline{\gamma.\infty}}, \quad |z| > 1, \iff \\
		& \iff \int_{G} \overline{f(\overline{(\gamma^{-1}.z)^{-1}})} \frac{z^2}{(\gamma^{-1}.z)^2} (\gamma^{-1})'(z) d\mu(\gamma) - \overline{f(\overline{z}^{-1})} = \int_G \frac{d \mu(\gamma)}{z^{-1} - \overline{\gamma.\infty}}, \quad |z| > 1, \iff \\
		& \iff \int_{G} \frac{\overline{f(\overline{(\gamma^{-1}.z)^{-1}})}}{(\gamma^{-1}.z)^2} (\gamma^{-1})'(z) d\mu(\gamma) - \frac{\overline{f(\overline{z}^{-1})} }{z^2} = \int_G \frac{1}{z} \frac{d \mu(\gamma)}{1 - z \overline{\gamma.\infty}} , \quad |z| > 1, \iff \\
		& \iff \int_{G} \frac{\overline{f(\overline{(\gamma^{-1}.z)^{-1}})}}{(\gamma^{-1}.z)^2} (\gamma^{-1})'(z) d\mu(\gamma) - \frac{\overline{f(\overline{z}^{-1})} }{z^2} = \int_G \frac{d \mu(\gamma) \overline{\gamma.\infty}}{1 - z \overline{\gamma.\infty}} + \frac{1}{z}, \quad |z| > 1, \iff \\
		& \iff \int_{G} \frac{\overline{f(\overline{(\gamma^{-1}.z)^{-1}})}}{(\gamma^{-1}.z)^2} (\gamma^{-1})'(z) d\mu(\gamma) - \frac{\overline{f(\overline{z}^{-1})} }{z^2} = \int_G \frac{d \mu(\gamma)}{\gamma.0 - z} + \frac{1}{z}, \quad |z| > 1.
	\end{aligned}
	\]
	All that is left to observe is that $\frac{1}{z}$ is mapped to $\int_G \left( \frac{1}{z - \gamma.0} - \frac{1}{z - \gamma.\infty} \right) d \mu(\gamma) - \frac{1}{z}$. Therefore, we see that $g(z)$ does, indeed, solve \eqref{main equation, simp}. 
\end{proof}

\textbf{Remark.} Observe that the transform in the above proposition, indeed, maps $f_\nu(z)$ on $\mathbb{D}$ to $f_\nu(z)$ on $\mathbb{D}_e$. We will leave the explicit formulation of Theorem \ref{converse} outside of the unit disk, but we will formulate the following useful lemma which is, essentially, a part of the proof of Theorem \ref{converse}.

\begin{lemma}
	Let $\mu$ be a countably supported probability measure whose support generates a lattice $\Gamma \leq G$, and assume that $(S^1, \nu)$ is the model for the Poisson boundary of $(\Gamma, \mu)$. Denote the (bounded) $\mu$-harmonic function on $\Gamma$ represented by $\xi \mapsto \frac{1}{\xi - z_0}$ by $\lambda_{z_0}$ for $z_0 \in \overline{\mathbb{C}} \setminus \mathbb{T}$. Then we have
	\begin{equation}
		\label{harmonic representation}
		f_\nu(z_0) = \frac{\lambda_{z_0}(\gamma) + \frac{1}{z_0 - \gamma.\infty}}{(\gamma^{-1})'(z_0)}, \quad z_0 \in \overline{\mathbb{C}} \setminus \mathbb{T}.
	\end{equation}
\end{lemma}

We will refer to \eqref{harmonic representation} as the \textbf{harmonic representation} of the hitting measure's Cauchy transform.

\section{Squeezing water from a stone: a deep dive into \eqref{main equation, simp}}
\label{Squeezing water from a stone}
In this section we will explore the functional equation \eqref{main equation, simp} in much more detail. From now on, we will restrict ourselves to countably supported probability measures $\mu$, denoting by $\Gamma \le G = PSU(1,1)$ the subgroup generated by the support of $\mu$.

\begin{theorem}[Corollary \ref{C:main corollary}.1]
	\label{entire solutions}
	Let $\mu$ be a probability measure with finite support. Then there are no entire solutions to \eqref{main equation, simp}.
\end{theorem}
\begin{proof}
	Choose an element $\tau \in \text{supp}(\mu)$ which does not fix the origin. In particular, $\tau.\infty = (\overline{\tau.0})^{-1} \ne \infty$. Fix small enough contour $C_\tau$ around $\tau.\infty$. Integrating both sides over this contour, we get
	\[
	\begin{aligned}
		& \int_{C_\tau}\left( \int_\Gamma \frac{d\mu(\gamma)}{z - \gamma.\infty}  + \int_\Gamma f(\gamma^{-1}.z)(\gamma^{-1})'(z)  d\mu(\gamma) - f(z) \right) dz = \\ 
		&= \sum_{\gamma.\infty = \tau.\infty} \mu(\gamma) + \int_\Gamma \int_{\gamma^{-1}(C_\tau)} f(z) dz - \int_{C_\tau} f(z)  = 0
	\end{aligned}
	\]
	by applying the change of variables. As $f(z)$ is entire, the contour integrals vanish, leaving us with $\mu(\gamma) = 0$ for all $\gamma$ with the same pole as $\tau$, which leads to a contradiction.
\end{proof}

\begin{corollary}
	Let $\mu$ be a probability measure with finite support. Then $\limsup\limits_{k \rightarrow \infty} |a_k|^{1/k} > 0$ for every $\mu$-stationary measure with the Fourier series $\nu \sim \sum_{k \in \mathbb{Z}} a_k e^{i k t}$.
\end{corollary}
\begin{proof}
	Consider a $\mu$-stationary measure $\nu$ with $\limsup_{k \rightarrow \infty} |a_k|^{1/k} = 0$. Then $f_\nu(z)$ is an entire function which solves \eqref{main equation, simp} for $|z| < 1$. The LHS of \eqref{main equation, simp} can be analytically continued to a meromorphic function on $\mathbb{C}$, therefore, $f_\nu(z)$ solves \eqref{main equation, simp} for all $\mathbb{C}$. This allows us to apply Theorem \ref{entire solutions}, yielding a contradiction.
\end{proof}
\begin{theorem}[Corollary \ref{C:main corollary}.2]
	Let $\mu$ be a countably supported probability measure. If \\ $\limsup\limits_{n \rightarrow \infty} \left\| \sum_\Gamma \frac{\mu^{*n}(\gamma)}{z - \gamma.\infty} \right\|_1 = \infty$, then there are no $\mu$-stationary measures with $L^{1+\varepsilon}(S^1, Leb)$-density for any $\varepsilon > 0$.
\end{theorem}
\begin{proof}
	Let $\nu$ be $\mu$-stationary with density in $L^{1+\varepsilon}(S^1)$. Due to Fatou's theorem we know that $f_\nu(z) \in H^{1 + \varepsilon}(\mathbb{D})$. In particular, $f_\nu(z) \in H^1(\mathbb{D})$. As all composition operators in LHS of \eqref{main equation, simp} are isometries, we can see that the $H^1$-norm of LHS is at most $2||f||_1$. Make note of the fact that this application of the triangle inequality does not depend on $\mu$ at all. Applying the $H^1$-norm to both sides, we get
	\[
	2 \left\| f_\nu \right\|_1 \ge \left\| \sum_\Gamma \frac{\mu(\gamma)}{z - \gamma.\infty} \right\|_1.
	\]
	However, keep in mind that any $\mu$-stationary measure is $\mu^{*n}$-stationary, therefore, WLOG one can replace $\mu$ with $\mu^{*n}$ in the above inequality without changing LHS. This would imply
	\[
	2 \left\| f_\nu \right\|_1 \ge \limsup\limits_{n \rightarrow \infty} \left\| \sum_\Gamma \frac{\mu^{*n}(\gamma)}{z - \gamma.\infty} \right\|_1 = \infty,
	\]
	which leads to a contradiction.
\end{proof}


\begin{example}
	Consider $\mu = \delta_\gamma$ for a non-elliptic $\gamma \in PSU(1,1)$. Then the $H^1$-norm of $\frac{1}{z - \gamma^n.\infty}$ goes to infinity as $n \rightarrow \infty$, so there are no absolutely continuous measures with densities in $L^{1 + \varepsilon}(S^1)$, as we expected.
\end{example}

However, as simple as this criterion seems, given a measure $\mu$ supported on a lattice in $PSU(1,1)$, it is not at all easy to estimate $\left\| \sum_\Gamma \frac{\mu^{*n}(\gamma)}{z - \gamma.\infty} \right\|_1$, therefore, a potential argument should rely on a very precise analysis of how non-uniformly the poles will be distributed in small neighbourhoods of $\mathbb{T}$.

Finally, the proof of Theorem \ref{converse} can be adapted to show the following theorem.

\begin{theorem}
	Let $\Gamma \leq G$ be a lattice equipped with a probability measure $\mu$. Assume that $(S^1, \nu)$ is the Poisson boundary for $(\Gamma, \mu)$. Consider the (unique) $\mu$-harmonic function $\lambda : \Gamma \rightarrow \mathbb{C}$ which converges to the identity function on $\mathbb{T}$. TFAE:
	\begin{itemize}
		\item  
		$
		\sup\limits_{\gamma \in \Gamma} \left| \frac{\lambda(\gamma) - \gamma.0}{(\gamma^{-1})'(0)} \right| < \infty.
		$
		\item $f_\nu(z) \in H^\infty(\mathbb{D})$.
	\end{itemize}
	In particular,
	\begin{itemize}
		\item The above supremum is finite implies that the hitting measure is absolutely continuous and the density belongs to $L^p$ for all $1 \le p < \infty$
		\item The above supremum is infinite implies that the hitting measure cannot have a continuous density.
	\end{itemize}
\end{theorem}
\begin{proof}
	\indent
	Before we start the proof, we consider the function
	\[
		\lambda_0: \gamma \mapsto f_\nu(\gamma^{-1}.0) (\gamma^{-1})'(0) + \frac{1}{\gamma.\infty} = f_\nu(\gamma^{-1}.0) (\gamma^{-1})'(0) + \overline{\gamma.0}.
	\]
	From the proof of Theorem \ref{converse} we see that this is a bounded harmonic function with the boundary representing function being $\xi \mapsto \overline{\xi}$. Therefore, $\lambda_0 = \overline{\lambda}$, and we have
		\[
	f_\nu(\gamma^{-1}.0) = \dfrac{\overline{\lambda}(\gamma) - \overline{\gamma.0}}{(\gamma^{-1})'(0)}, \quad \gamma \in \Gamma.
	\]
	Now all we need is to recall that $\Gamma.z_0$ is non-tangentially dense in $\mathbb{D}$ and does not have interior limit points due to the fact that $\Gamma$ is a lattice. As Cauchy transforms have non-tangential limits $Leb$-almost everywhere, uniform boundedness on the orbit will force $f_\nu$ to be in $H^\infty(\mathbb{D})$. See \cite[Theorem 3]{brownsums} for a related stronger statement.
\end{proof}

\subsection{Functional-analytic necessary condition for existence of absolutely continuous stationary measures}

In this subsection we treat LHS of \eqref{main equation, simp} as a bounded operator: define

\[
T_\mu: H^p(\mathbb{D}) \rightarrow H^p(\mathbb{D}), \quad T_\mu(f)(z) := \sum_{\gamma} \mu(\gamma) f(\gamma^{-1}.z)(\gamma^{-1})'(z) - f(z).
\]

It is well-known that $T_\mu$ is a bounded operator for all $0 < p \le \infty$, and in such generality, not much else is known about $T_\mu$. If $p > 1$, we recall that $(H^p)^* \cong H^q$, for $\frac{1}{p} + \frac{1}{q} = 1$, and we can at least explicitly compute its adjoint $T^*_\mu : H^q \rightarrow H^q$.

\begin{proposition}
	Consider $V_\gamma(f)(z) = (f \circ \gamma^{-1})(\gamma^{-1})'(z)$ as a bounded operator $H^p \rightarrow H^p$. Then
	\[
	V^*_\gamma(f)(z) = S^*(f(\gamma.z) \gamma.z), \quad f \in H^q(\mathbb{D}),
	\]
	where $S^*$ stands for the backwards shift $S^*(g)(z) = \frac{g(z) - g(0)}{z}$.
\end{proposition}
\begin{proof}
	As in many similar computations (see \cite[Theorem 2]{cowen1988linear} for an example), we use the reproducing kernel property of $\frac{1}{1 - \overline{a}z}$: for any $f \in H^p$
	\[
	\left\langle f(z), \frac{1}{1 - \overline{a}z} \right\rangle := \frac{1}{2\pi} \int_{0}^{2 \pi} \frac{f(e^{it}) dt}{1 - a \overline{z}} = f(a).
 	\]
 	A slight modification yields
 	\[
 	\left\langle f(z), \frac{1}{a - z} \right\rangle = \frac{f(\overline{a}^{-1})}{\overline{a}}.
 	\]
 	As we know how $T_\gamma$ acts on $\frac{1}{a - z}$, reflexivity of $H^p$ for all $1 < p < \infty$ allows us to write
 	\[
 	\begin{aligned}
 		\frac{V^*_\gamma f(\overline{a}^{-1})}{\overline{a}} &= \left\langle (V_\gamma^*)f(z), \frac{1}{a - z} \right\rangle =  \left\langle f(z), \frac{1}{\gamma.a - z} - \frac{1}{\gamma.\infty - z} \right\rangle = \\ 
 		&=\frac{f(\overline{\gamma.a}^{-1})}{\overline{\gamma.a}} - \frac{f(\overline{\gamma.\infty}^{-1})}{\overline{\gamma.\infty}} \stackrel{\eqref{from inside to outside}}{=} \gamma.\overline{a^{-1}} f(\gamma.\overline{a^{-1}}) - \gamma.0 f(\gamma.0).
 	\end{aligned}
 	\]
 	Replacing $\overline{a}^{-1}$ with $\omega$, we get
 	\[
 	V^*_\gamma f(w) = \frac{\gamma.w f(\gamma.w) - \gamma.0 f(\gamma.0)}{w} = S^*(f(\gamma.w) \gamma.w).
 	\]
\end{proof}

As a quick corollary, we get that

\begin{equation}
	\label{adjoint of LHS}
	T_\mu^*(f)(z) = S^*\left( \sum_{\gamma} \mu(\gamma) f(\gamma.z) \gamma.z\right)  - f(z).
\end{equation}

\begin{theorem}
	\label{Blaschke and orthogonals}
	Let $\mu$ satisfy the Blaschke condition, and let $1 < p < \infty$. Consider a $\mu$-stationary measure $\nu$ with $L^p$-density. Then the following statements are equivalent.
	\begin{itemize}
		\item The Cauchy transform $f_\nu$ solves \eqref{main equation, simp},
		\item $f_\nu(z) \in (\overline{T_\mu^*(B_\mu H^q)})^{\perp}$,
	\end{itemize}
	where $B_\mu(z)$ is the Blaschke product vanishing on the support of $\mu$.
\end{theorem}

\begin{proof}
	Let $T_\mu(f) = \sum_{\gamma} \frac{\mu(\gamma)}{z - \gamma.\infty}$. It is easy to see that $\sum_{\gamma} \frac{\mu(\gamma)}{z - \gamma.\infty}$ is a linear combination of reproducing kernels $\frac{1}{1 - \overline{\gamma.0}z}$. In particular, $\sum_{\gamma} \frac{\mu(\gamma)}{z - \gamma.\infty} \in (B_\mu H^q)^\perp$. Therefore,
	\[
	0 = \left\langle T_\mu(f), B_\mu H^q \right\rangle = \left\langle f, \overline{T_\mu^*(B_\mu H^q)} \right\rangle.
	\]
	
	To prove the reverse implication, we use the structure of $S^*$-invariant subspaces in $H^p$ to show 
	$$
	T_\mu(f_\nu) = \sum_{\gamma \in \text{supp} \, \mu} \frac{a_{\gamma}}{1 - \overline{\gamma.0} z} := \sum_{\gamma \in \text{supp} \, \mu} \frac{b_{\gamma}}{\gamma.\infty - z},
	$$
	and then we can use the residue theorem to show that $b_\gamma = \mu(\gamma)$, as the Blaschke condition guarantees that all poles are, indeed, isolated.
\end{proof}

This proves Corollary \ref{functional-analytic necessary condition}.
\subsection{When the Lebesgue measure is stationary?}

Earlier we have reproved the well-known theorem of J. Bourgain that the Lebesgue measure cannot be a stationary measure if $\mu$ has finite support. To understand this case better, we need to look at \eqref{main equation, simp} and observe that $f_\nu$ vanishes, leaving us with vanishing of the following \textbf{Borel series}.
\begin{equation}
	\label{Lebesgue equation}
	\sum_{\gamma} \frac{\mu(\gamma)}{z - \gamma.\infty} = 0, \quad |z| < 1.
\end{equation}
This immediately proves Corollary \ref{C:main corollary}.1. At first glance, it might seem counter-intuitive that the above sum can vanish on the entire disc, but let us recall the following fundamental fact about Borel series.

\begin{theorem}[\cite{brownsums}, Theorem 3]
	\label{non-tangential equiv rep}
	Let $A = \{ z_n \} \subset \mathbb{D}$ be a sequence of points \textbf{inside} the unit disk without interior limit points. Then there exists a sequence $\{c_n\} \in l^1$ such that
	\[
	\sum_n \frac{c_n}{z - z_n} = 0, \quad |z| > 1
	\]
	if and only if almost every point in $S^1$ is a non-tangential limit of a subsequence in $\{z_n\}$.
\end{theorem}

This theorem almost gives what we want, however, the above theorem gives series with poles \textbf{inside} the disk which vanishes \textbf{outside} of it, whereas we need the opposite: Borel series with poles \textbf{outside} the disk and vanishing \textbf{inside} the disk.

One can easily mitigate this by considering the change of variables $z \mapsto 1/z$:
\[
\begin{aligned}
	\sum_{\gamma} \frac{\mu(\gamma)}{z^{-1} - \gamma.\infty} &= \sum_{\gamma} \frac{\mu(\gamma) z}{1 - (\gamma.\infty) z} = \sum_{\gamma} \frac{(\gamma.\infty)^{-1}  \mu(\gamma) z}{(\gamma.\infty)^{-1} -  z} = \\ 
	&= \sum_{\gamma} - \frac{\mu(\gamma)}{\gamma.\infty} + \frac{\mu(\gamma)}{(\gamma.\infty)^2} \frac{1}{(\gamma.\infty)^{-1} - z}.
\end{aligned}
\]
However, as we can plug in $z = 0$ in \eqref{Lebesgue equation}, we get that
\[
\sum_{\gamma} \frac{\mu(\gamma)}{z^{-1} - \gamma.\infty} = \sum_{\gamma}\frac{\mu(\gamma)}{(\gamma.\infty)^2} \frac{1}{(\gamma.\infty)^{-1} - z} = 0
\]
for all $|z| > 1$. Applying the Brown-Shields-Zeller theorem, we obtain Corollary \ref{intro: Lebesgue is stationary}.3.

\textbf{Remark.} Recall that the orbit of a point with respect to an action of a discrete subgroup of $PSU(1,1)$ is non-tangentially dense if and only if the subgroup is of the first type. Therefore, Theorem \ref{non-tangential equiv rep} confirms that $\Gamma \subset PSU(1,1)$ being a first-kind Fuchsian group should be a necessary condition for a Furstenberg measure on $\Gamma$ to exist.

Moreover, due to another theorem of Beurling, referring to \cite[Corollary 4.2.24]{Shapiro1968}:

\begin{theorem}[\cite{beurling1934fonctions}, \cite{beurling1989collected}]
	Let $\{ z_n \}$ be a sequence of points \textbf{outside} of the unit disk with $|z_n| \downarrow 1$. If $\limsup\limits_{n \rightarrow \infty} |c_n|^{1/n} < 1$ and
	\[
	\sum_n \frac{c_n}{z - z_n} = 0, \quad |z| < 1,
	\]
	then all $c_n = 0$.
\end{theorem}
Applying this theorem to $z_n = \gamma_n.\infty$ (relative to a suitable enumeration of $\Gamma$), we get that a Fuchsian group of first kind $\Gamma \subset PSU(1,1)$ admits a Furstenberg measure only if
\[
\limsup\limits_{n \rightarrow \infty} |\mu(\gamma_n)|^{1/n} = 1,
\] 
thus proving Corollary \ref{intro: Lebesgue is stationary}.2. Combined with the exponential growth of Fuchsian groups, this condition implies that a Furstenberg measure $\mu$ cannot have a double-exponential moment with respect to the hyperbolic distance: if we let $c > 0$, then
\[
\sum_{n} \mu(\gamma_n) e^{e^{c d(0, \gamma_n.0)}} < \infty \iff \sum_{n} \mu(\gamma_n) e^{c n} < \infty \Rightarrow \limsup\limits_{n \rightarrow \infty} |\mu(\gamma_n)|^{1/n} < e^{-c} < 1.
\]
As for the strongest known moment conditions: it is known that Jialun Li's counterexample, given in the Appendix of \cite{10.1215/00127094-2020-0058}, provides a Furstenberg measure with an exponential moment, our approach shows that a Furstenberg measure cannot have a double-exponential moment. It is widely believed for cocompact lattices that there is a Furstenberg measure with a superexponential moment, but we are not aware of a complete and self-contained argument being published. Moreover, we would like to mention several existing results related to decay of the coefficients in Borel series.
\begin{itemize}
	\item Due to Denjoy, \cite{denjoy}, there exist $c_n$ and $z_n \in \mathbb{D}$, with $|c_n| \le k e^{-n^{1/2} - \varepsilon}$ such that
	\[
		\sum_n \frac{c_n}{z - z_n} = 0, \quad |z| > 1.
	\]
	\item Another example due to Beurling, \cite{beurvol1}: there exist $c_n$ and $z_n \in \mathbb{D}$, with $|c_n| \le k e^{-\frac{n}{\log(n)^2}}$ such that
	\[
		\sum_n \frac{c_n}{z - z_n} = 0, \quad |z| > 1.
	\]
	\item Finally, due to Leont'eva \cite{leonteva} we have a very strong result: for any function which is holomorphic in a \textbf{closed} disk $\overline{\mathbb{D}}$ there exist $|z_n| > 1$ and $A_n$ with $|A_n| \le C e^{-n^{1 - \varepsilon_n}}$ for some sequence $\varepsilon_n \rightarrow 0$ such that
	\[
		\sum_n \frac{A_n}{z - z_n} = f(z), \quad |z| < 1.
	\]
\end{itemize}
Unfortunately, none of these results provide control on the positions of the poles $z_n$, but at least this suggests another heuristic for the possibility of a Furstenberg measure with a superexponential moment with respect to the hyperbolic distance.

\section{Criterion for the singularity of the harmonic measure}
\label{Criterion for the singularity of the harmonic measure}
The following theorem is an immediate corollary of the Fatou's theorem:
\begin{theorem}
	\label{first equiv}
	For a finite complex Borel measure on $\mathbb{T}$ the following statements are equivalent.
	\begin{enumerate}
		\item The hitting measure $\mu$ is singular.
		\item The exterior ($|z| > 1$) part of $f_\nu(z)$ is a \textbf{pseudocontinuation} of the inner $(|z| < 1)$ part $f_\nu(z)$. In other words, the non-tangential limits exist and coincide $Leb$-a.s.
		\item The harmonic (wrt to the Euclidean/hyperbolic Laplacian) function
		\[
		h(z) := \frac{1}{2\pi} \int_0^{2\pi} \frac{1 - |z|^2}{|e^{it} - z|^2} d\nu(t)
		\]
		has $Leb$-a.e. vanishing non-tangential limits at the unit circle.
	\end{enumerate}
\end{theorem}

As we care about discrete subgroups, we would like to make use of both the second and the third criteria via \eqref{harmonic representation}.

\begin{theorem}
	Let $\mu$ be a probability measure on a discrete subgroup $\Gamma$ of $PSU(1,1)$, and assume that $(S^1, \mu)$ is the model for the Poisson boundary of $(\Gamma, \mu)$. Then the following statements are equivalent.
	\begin{enumerate}
		\item For $Leb$-almost all $\xi \in \mathbb{T}$ we have
		\[
		\lim_{r \rightarrow 1^-} \lim_{n \rightarrow \infty} \sum_{\gamma} \mu^{* n}(\gamma) \frac{1 - r^2}{|\gamma(0) - r\xi|^2} = 0.
		\]
		\item There exist $|z| < 1$ and $|w| > 1$ such that for $Leb$-a.s. $\xi \in S^1$ the following non-tangential limits exist and are equal to each other:
		\[
		\angle \lim\limits_{\gamma.z \rightarrow \xi} \frac{\lambda_z(\gamma) - \frac{1}{z - \gamma.\infty}}{(\gamma^{-1})'(z)} = \angle \lim\limits_{\gamma.w \rightarrow \xi} \frac{\lambda_w(\gamma) - \frac{1}{w - \gamma.\infty}}{(\gamma^{-1})'(w)}.
		\]
		\item For every $|w| > 1$ and the following non-tangential limits exist and are equal to each other for $Leb$-a.s. $\xi \in S^1$:
		\[
		\angle \lim\limits_{\gamma.0 \rightarrow \xi} \frac{\lambda_0(\gamma) - \overline{\gamma.0}}{(\gamma^{-1})'(0)} = \angle \lim\limits_{\gamma.w \rightarrow \xi} \frac{\lambda_w(\gamma) - \frac{1}{w - \gamma.\infty}}{(\gamma^{-1})'(w)}.
		\]
	\end{enumerate}
\end{theorem}

\begin{proof}
	(2) is equivalent to Theorem \ref{first equiv}'s (1) because $\mu^{*n} \xrightarrow{wk^*} \nu$. (2) and (3) are equivalent to Theorem \ref{first equiv}'s (3) because of the harmonic representation:
	\[
	\frac{\lambda_z(\gamma) - \frac{1}{z - \gamma.\infty}}{(\gamma^{-1})'(z)} = f_\nu(\gamma^{-1}.z), \quad \frac{\lambda_w(\gamma) - \frac{1}{w - \gamma.\infty}}{(\gamma^{-1})'(w)} = f_\nu(\gamma^{-1}.w),
	\]
	and the above identities are equivalent to the equality between the inner and outer non-tangential limits of $f_\nu$ almost everywhere with respect to the Lebesgue measure.
\end{proof}

The above theorem suggests that one should be looking for intrinsic (with respect to a lattice) construction of harmonic functions, which would allow us to control the convergence to be able to estimate the above ratios effectively.

\section{Open questions}
\begin{itemize}
	\item It is easy to see from the proof of Corollary \ref{C:main corollary}.1 that we actually get non-existence of solutions $f(z) = \sum a_{k+1} z^k$ to \eqref{main equation, simp} with $\limsup_{n \rightarrow \infty} |a_k|^{1/k} < \varepsilon$ for some small $\varepsilon$, as only one preimage of the chosen contour explodes, so we can bound the radius of the convergence of the solution. Ideally, one would like to show that for finitely supported $\mu$ every solution of \eqref{main equation, simp} has radius of convergence exactly $1$. Keep in mind that this result would almost close the smoothness gap: it is known that absolutely continuous densities stationary with respect to finitely supported measures can belong to $C^n(S^1)$ for any $n > 1$. 
	
	The Douglas-Shields-Shapiro theorem implies that any holomorphic function with the radius of convergence exceeding $1$ is either cyclic with respect to the backward shift or rational. It is reasonable to assume that \eqref{main equation, simp} only has rational solutions when $\mu$ is supported on a single element, and we conjecture that former never happens.
	\item The Brown-Shields-Zeller theorem has an unexpected consequence -- it requires the poles to be non-tangentially dense \textbf{almost} everywhere on $S^1$. Therefore, even if $\Gamma$ is a non-cocompact lattice, there will be a sequence $(a_\gamma) \in l^1(\Gamma)$ such that
	\[
	\sum \frac{a_\gamma}{z - \gamma.\infty} = 0, \quad |z| < 1.
	\]
	However, due to \cite{guivarch1990} we know that the Lebesgue measure is not stationary with respect to any $\mu$ with finite first moment. Therefore, either Theorem \ref{T:main result} is not a criterion, $(a_\gamma)$ does not have the first finite moment, or there is a complex-valued Furstenberg measure -- keep in mind that Guivarch'-le Jan's methods only apply for probability measures $\mu$.
	\item Cauchy transforms were used to study affine self-similar measures on $\mathbb{C}$ in \cite{lund1998cauchy}. However, the paper was focused on studying measures supported on fractals with the Hausdorff dimension $\alpha> 1$, which forces the Cauchy transforms to be bounded and Holder with exponent $\alpha - 1$ (see \cite[Theorem 2.1(b)]{lund1998cauchy}). It should be possible to characterize the Hausdorff dimension of hitting measures using a similar self-similarity condition for the non-tangential limit $f_\nu(e^{it})$, but we expect the Hausdorff dimensions be strictly smaller than $1$, and we are not aware of any existing ideas in this direction.
	\item Another difficulty which stands in our way of resolving the singularity conjecture is the fact that the operator $T_\mu: f \mapsto \sum \mu(\gamma) f \circ (\gamma^{-1}) (\gamma^{-1})'$ does not commute with the backward shift $S^* : f \mapsto \frac{f(z) - f(0)}{z}$. There is a very tempting approach that consists of using the structure of $S^*$-invariant subspaces of $H^p$ as follows:
	\begin{enumerate}
		\item The first and more manageable step is to show that for finitely supported measures $\mu$ on a lattice $\Gamma$ there are no \textbf{non-cyclic} solutions $f \in (\varphi H^p)^\perp$, for $p > 1$. The key idea is to use the Douglas-Shields-Shapiro theorem and show that the pseudocontinuation of $f$ also solves \eqref{main equation, simp}, but also has to have non-trivial poles. We can flip this solution to get a solution with singularities inside the unit disk, and then, using the action of $\Gamma$, show that poles would not satisfy the Blaschke condition, contradicting the fact that $\varphi$ is inner.
		\item Carefully dealing with solutions $f \in H^1(\mathbb{D}) \setminus H^{1 + \varepsilon}(\mathbb{D})$, we would be able to show that either the harmonic measure is singular, or its Cauchy transform is cyclic with respect to the backward shift in $H^p$ for all $0 < p < \infty$.
		\item The most difficult part is eliminating the second possibility. There are multiple ways to potentially get a breakthrough.
		\begin{itemize}
			\item It is not difficult to see that the iterates $\sum \frac{\mu(\gamma)}{z - \gamma.\infty}$ converge to $f_\mu(z)$ on compact subsets away from the poles and $\mathbb{T}$. However, we would like to show that this convergence holds in all $H^p(\mathbb{D})$, $p < 1$, and we would like to use it to approximate $f_\nu(z)$ by finite convex combinations of simple poles $\frac{1}{z - \xi}$, $\xi \in \mathbb{T}$ in $H^p$ as well. Then we would force $f_\nu \in H^p \cap \overline{H^p_0}$ by Aleksandrov, and Fatou's theorem would yield singularity. The big problem with this approach is that it is notoriously tricky to estimate $L^p$-distances for $p < 1$, there are very few tools available to us, unless we switch to the upper half-plane $\{\text{Im} \, z > 0\}$, where we can try to borrow tools used to study real-variable Hardy spaces and Hardy distributions. This is, essentially, the core part of Aleksandrov's argument in \cite{Ale79}, which we were not able to adapt to our problem.
			\item The second approach is to somehow identify a non-trivial closed $S^*$-invariant subspace of $H^p$ which contains $\sum \frac{\mu(\gamma)}{z - \gamma.\infty}$, with its preimage with respect to $T_\mu$ also being a non-trivial closed $S^*$-invariant subspace. Once again, this would eliminate the option of $f_\nu$ being cyclic, delivering singularity. As we mentioned before, the operator $T_\mu$ does not commute with the backward shift, making this approach very tricky to implement as well. It is also not clear how to reconstruct a suitable pseudocontinuation for a potential $f$ even if we know that $T_\mu(f)$ has one.
		\end{itemize}
	\end{enumerate}
\end{itemize}

\bibliographystyle{alpha}
\bibliography{fourier_singular}

\begin{thebibliography}{CFFT22}

\bibitem[Ale79]{Ale79}
A.B. Aleksandrov.
\newblock Invariant subspaces of the backward shift operator in the space $h^p$
  ($p\in(0,1)$).
\newblock {\em Zap. Nauchn. Sem. LOMI}, 92, 1979.

\bibitem[BC89]{beurling1989collected}
Arne Beurling and Lennart Carleson.
\newblock {\em The collected works of Arne Beurling}, volume~1.
\newblock Springer, 1989.

\bibitem[Beu34]{beurling1934fonctions}
Arne Beurling.
\newblock Sur les fonctions limites quasi analytiques des fractions
  rationnelles.
\newblock In {\em 8th Scandinavian Math. Congress, Stockholm}, pages 199--210,
  1934.

\bibitem[Beu89]{beurvol1}
Arne Beurling.
\newblock {\em The collected works of {A}rne {B}eurling. {V}ol. 1}.
\newblock Contemporary Mathematicians. Birkh\"auser Boston, Inc., Boston, MA,
  1989.
\newblock Complex analysis.

\bibitem[BHM11]{blachere2011harmonic}
S{\'e}bastien Blach{\`e}re, Peter Ha{\"i}ssinsky, and Pierre Mathieu.
\newblock Harmonic measures versus quasiconformal measures for hyperbolic
  groups.
\newblock In {\em Annales scientifiques de l'{\'E}cole Normale Sup{\'e}rieure},
  volume~44, pages 683--721, 2011.

\bibitem[Bou12]{Bourgain2012}
Jean Bourgain.
\newblock {\em Finitely supported measures on {$SL_2(\Bbb R)$} which are
  absolutely continuous at infinity}, volume 2050 of {\em Lecture Notes in
  Math.}, pages 133--141.
\newblock Springer, Heidelberg, 2012.

\bibitem[BPS12]{MR2969625}
B.~B\'{a}r\'{a}ny, M.~Pollicott, and K.~Simon.
\newblock Stationary measures for projective transformations: the {B}lackwell
  and {F}urstenberg measures.
\newblock {\em J. Stat. Phys.}, 148(3):393--421, 2012.

\bibitem[BSZ60]{brownsums}
Leon Brown, Allen Shields, and Karl Zeller.
\newblock On absolutely convergent exponential sums.
\newblock {\em Transactions of the American Mathematical Society},
  96(1):162--183, 1960.

\bibitem[BW89]{bonsall1989vanishing}
FF~Bonsall and D~Walsh.
\newblock Vanishing l1-sums of the poisson kernel, and sums with positive
  coefficients.
\newblock {\em Proceedings of the Edinburgh Mathematical Society},
  32(3):431--447, 1989.

\bibitem[CFFT22]{chawla2022poissonboundaryhyperbolicgroups}
Kunal Chawla, Behrang Forghani, Joshua Frisch, and Giulio Tiozzo.
\newblock The poisson boundary of hyperbolic groups without moment conditions,
  2022.

\bibitem[Cim00]{cimahardy}
Joseph~A. Cima.
\newblock {\em The Backward Shift on the Hardy Space}.
\newblock Mathematical Surveys and Monographs. American Mathematical Society,
  2000.

\bibitem[CMR06]{book:738388}
Joseph~A. Cima, Alec~L. Matheson, and William~T. Ross.
\newblock {\em The Cauchy Transform}.
\newblock American Mathematical Society, 2006.

\bibitem[Cow88]{cowen1988linear}
Carl~C Cowen.
\newblock Linear fractional composition operators on h2.
\newblock {\em Integral equations and operator theory}, 11(2):151--160, 1988.

\bibitem[CT22]{cantrell2022invariant}
Stephen Cantrell and Ryokichi Tanaka.
\newblock Invariant measures of the topological flow and measures at infinity
  on hyperbolic groups, 2022.

\bibitem[Den24]{denjoy}
A.~Denjoy.
\newblock Sur les s\'eries de fractions rationnelles.
\newblock {\em Bull. Soc. Math. France}, 52:418--434, 1924.

\bibitem[Fur63]{furstenberg1963noncommuting}
Harry Furstenberg.
\newblock Noncommuting random products.
\newblock {\em Trans. Amer. Math. Soc.}, 108:377--428, 1963.

\bibitem[Fur71]{furstenberg71}
Harry Furstenberg.
\newblock Random walks and discrete subgroups of {L}ie groups.
\newblock In {\em Advances in {P}robability and {R}elated {T}opics, {V}ol. 1},
  pages 1--63. Dekker, New York, 1971.

\bibitem[GL23]{garcía2023dimension}
Ernesto García and Pablo Lessa.
\newblock Dimension drop of harmonic measure for some finite range random walks
  on fuchsian schottky groups, 2023.

\bibitem[GLJ90]{guivarch1990}
Yves Guivarc'h and Yves Le~Jan.
\newblock Sur l'enroulement du flot g\'{e}od\'{e}sique.
\newblock {\em C. R. Acad. Sci. Paris S\'{e}r. I Math.}, 311(10):645--648,
  1990.

\bibitem[HL90]{hayman1990bases}
Walter~K Hayman and Terry~J Lyons.
\newblock Bases for positive continuous functions.
\newblock {\em Journal of the London Mathematical Society}, 2(2):292--308,
  1990.

\bibitem[HW17]{axioms6020007}
John~E. Herr and Eric~S. Weber.
\newblock Fourier series for singular measures.
\newblock {\em Axioms}, 6(2), 2017.

\bibitem[JP98]{denseanalytic}
P.E.T. Jorgensen and S.~Pedersen.
\newblock Dense analytic subspaces in fractal $l^2$-spaces.
\newblock {\em J. Anal. Math.}, 75:185--228, 1998.

\bibitem[Kai00]{kaimanovich2000poisson}
Vadim~A. Kaimanovich.
\newblock The {P}oisson formula for groups with hyperbolic properties.
\newblock {\em Ann. of Math. (2)}, 152(3):659--692, 2000.

\bibitem[Kit23]{kittle2023absolutely}
Samuel Kittle.
\newblock Absolutely continuous furstenberg measures, 2023.

\bibitem[KiV83]{KV83}
V.~A. Ka\u~imanovich and A.~M. Vershik.
\newblock Random walks on discrete groups: boundary and entropy.
\newblock {\em Ann. Probab.}, 11(3):457--490, 1983.

\bibitem[KLP11]{kaimanovich2011matrix}
Vadim~A. Kaimanovich and Vincent Le~Prince.
\newblock Matrix random products with singular harmonic measure.
\newblock {\em Geom. Dedicata}, 150:257--279, 2011.

\bibitem[Kog22]{kogler2022locallimittheoremrandom}
Constantin Kogler.
\newblock Local limit theorem for random walks on symmetric spaces, 2022.

\bibitem[Kos23]{mythesis}
Petr Kosenko.
\newblock {\em Harmonic measures for random walks on cocompact Fuchsian
  groups}.
\newblock Phd thesis, University of Toronto, Toronto, ON, May 2023.
\newblock Available at \url{https://hdl.handle.net/1807/129800}.

\bibitem[Leo67]{leonteva}
T.A. Leont'eva.
\newblock Representation of analytic functions by series of rational functions.
\newblock {\em Mathematical Notes of the Academy of Sciences of the USSR},
  2:695–702, 1967.

\bibitem[Leq22]{lequen2022absolutely}
Félix Lequen.
\newblock Absolutely continuous furstenberg measures for finitely-supported
  random walks, 2022.

\bibitem[LNP21]{10.1215/00127094-2020-0058}
Jialun Li, Fr{\'e}d{\'e}ric Naud, and Wenyu Pan.
\newblock {Kleinian Schottky groups, Patterson–Sullivan measures, and Fourier
  decay}.
\newblock {\em Duke Mathematical Journal}, 170(4):775 -- 825, 2021.

\bibitem[LSV98]{lund1998cauchy}
John-Peter Lund, Robert~S Strichartz, and Jade~P Vinson.
\newblock Cauchy transforms of self-similar measures.
\newblock {\em Experimental Mathematics}, 7(3):177--190, 1998.

\bibitem[Pri56]{privalov1956randeigenschaften}
I.I. Privalov.
\newblock {\em Randeigenschaften analytischer Funktionen}.
\newblock Hochschulb{\"u}cher f{\"u}r Mathematik. Deutscher Verlag der
  Wissenschaften, 1956.

\bibitem[RS02]{Shapiro1968}
William~T. Ross and Harold~S. Shapiro.
\newblock {\em Generalized analytic continuation}.
\newblock American Mathematical Society, 2002.

\bibitem[Str90]{strichartzI}
Robert~S. Strichartz.
\newblock Self-similar measures and their fourier transforms i.
\newblock {\em Indiana University Mathematics Journal}, 39(3):797--817, 1990.

\bibitem[Web17]{weber2017paleywiener}
Eric~S. Weber.
\newblock A paley-wiener type theorem for singular measures on $\mathbb{T}$,
  2017.

\end{thebibliography}
\Addresses
\end{document}